\documentclass[11pt]{amsart}
\usepackage{amsmath,amssymb,amsthm,mathtools,enumitem,graphicx,eucal,enumitem,cancel,xcolor}
\usepackage[top=30mm, bottom=35mm, left=30mm, right=30mm]{geometry}
\DeclareFontEncoding{OT2}{}{}
\DeclareFontSubstitution{OT2}{cmr}{m}{n}
\DeclareFontFamily{OT2}{cmr}{\hyphenchar\font45 }
\DeclareFontShape{OT2}{cmr}{m}{n}{
<5><6><7><8><9>gen*wncyr
<10><10.95><12><14.4><17.28><20.74><24.88>wncyr10}{}
\DeclareFontShape{OT2}{cmr}{b}{n}{
<5><6><7><8><9>gen*wncyb
<10><10.95><12><14.4><17.28><20.74><24.88>wncyb10}{}
\DeclareMathAlphabet{\mathcyr}{OT2}{cmr}{m}{n}
\DeclareMathAlphabet{\mathcyb}{OT2}{cmr}{b}{n}
\SetMathAlphabet{\mathcyr}{bold}{OT2}{cmr}{b}{n}
\newtheorem{theoremcounter}{Theorem Counter}[section]
\theoremstyle{plain}
\newtheorem{theorem}[theoremcounter]{Theorem}
\newtheorem{conjecture}[theoremcounter]{Conjecture}
\newtheorem{lemma}[theoremcounter]{Lemma}
\newtheorem{corollary}[theoremcounter]{Corollary}
\theoremstyle{definition}
\newtheorem{definition}[theoremcounter]{Definition}
\newcommand{\cA}{\mathcal{A}}
\newcommand{\bc}{\boldsymbol{c}}
\newcommand{\bdelta}{\boldsymbol{\delta}}
\newcommand{\cD}{\mathcal{D}}
\newcommand{\fD}{\mathfrak{D}}
\newcommand{\cH}{\mathcal{H}}
\newcommand{\NN}{\mathbb{N}}
\newcommand{\cP}{\mathcal{P}}
\newcommand{\QQ}{\mathbb{Q}}
\newcommand{\RR}{\mathbb{R}}
\newcommand{\ZZ}{\mathbb{Z}}
\newcommand{\wt}{\mathrm{wt}}
\numberwithin{equation}{section}

\allowdisplaybreaks
\newcommand{\sh}{\mathbin{\mathcyr{sh}}}
\DeclareMathOperator{\Ker}{Ker}
\newcommand{\smallxcancel}[2][0.8]{\mathrlap{{\renewcommand{\CancelColor}{\color{red}}\scalebox{#1}{$\xcancel{\phantom{#2}}$}}}#2}
\address{Nagahama Institute of Bio-Science and Technology, 1266, Tamura, Nagahama, Shiga, 526-0829, Japan}
\email{s\_seki@nagahama-i-bio.ac.jp}
\thanks{This research was supported by JSPS KAKENHI Grant Number JP21K13762.}
\title[]{Diamond lift of Hirose--Sato's formula involving the Hoffman basis}
\author[]{Shin-ichiro Seki}
\date{}
\begin{document}
\begin{abstract}
In this paper, we give a new proof of Hirose--Sato's formula for the expansion of
\[
\zeta(\{2\}^{a_1-1},3,\dots,\{2\}^{a_r-1},3,\{2\}^{c-1},1,\{2\}^{b_1},\dots, 1,\{2\}^{b_s})
\]
in the Hoffman basis, using the drop 1 relation.
\end{abstract}
\maketitle
\section{Introduction}
For positive integers $k_1, \dots, k_r$ with $k_r\geq 2$, the \emph{multiple zeta value} (MZV) $\zeta(k_1,\dots, k_r)$ is defined by
\[
\zeta(k_1,\dots, k_r)\coloneqq\sum_{0<n_1<\cdots <n_r}\frac{1}{n_1^{k_1}\cdots n_r^{k_r}}.
\]
Linear relations among MZVs have been studied for many years.
For each positive integer $c$, the relation
\begin{equation}\label{eq:Hoffman-conj}
\zeta(3,\{2\}^{c-1},1,2)=\zeta(\{2\}^{c+2})+2\zeta(3,3,\{2\}^{c-1})
\end{equation}
is known as \emph{Hoffman's conjectural identity} and had remained open for about seventeen years.
Here the notation $\{2\}^n$ denotes the string $2,\dots, 2$ in which the entry 2 is repeated $n$ times.
Charlton proved in \cite{Charlton2016} that there exists a rational number $q_c$ such that $\zeta(3,\{2\}^{c-1},1,2)=q_c\zeta(\{2\}^{c+2})+2\zeta(3,3,\{2\}^{c-1})$, and Hirose and Sato proved not only Hoffman's conjectural identity but also the following more general theorem.
\begin{theorem}[Hirose--Sato~\cite{HiroseSato2019A}]\label{thm:HiroseSato1}
For positive integers $a, b$, and $c$, we have
\begin{align*}
&\zeta(\{2\}^{a-1},3,\{2\}^{c-1},1,\{2\}^b)\\
&=\zeta(\{2\}^{a+b+c})+\zeta(\{2\}^{a-1},3,\{2\}^{b-1},3,\{2\}^{c-1})+\zeta(\{2\}^{b-1},3,\{2\}^{a-1},3,\{2\}^{c-1}).
\end{align*}
\end{theorem}
Their proof can be viewed as showing that the above identity is a consequence of the \emph{confluence relation} later formulated in \cite{HiroseSato2019B}.

Hirose and Sato gave a further generalization in \cite{HiroseSato2022+}, and we set up notation in order to state their result.
Let $\cH\coloneqq\ZZ\langle x,y\rangle$ be the non-commutative polynomial ring over $\ZZ$ in two variables $x$ and $y$.
Let $z_k\coloneqq yx^{k-1}$ for each positive integer $k$.
Subrings $\cH^{2,3}\subset\cH^0\subset\cH$ are defined by $\cH^{2,3}\coloneqq\ZZ\langle z_2,z_3\rangle$ and $\cH^0\coloneqq\ZZ+y\cH x$, respectively.
A $\ZZ$-linear map $Z\colon\cH^0\to\RR$ is defined by $Z(z_{k_1}\cdots z_{k_r})=\zeta(k_1,\dots, k_r)$ and $Z(1)=1$.
A product $\star$ on $\cH^{2,3}$ is defined inductively by 
\begin{itemize}[leftmargin=2.5em]
\item $w\star 1=1\star w=w$,
\item $wz_2\star w'=w\star w'z_2=(w\star w')z_2$,
\item $wz_3\star w'z_3=(w\star w'z_3)z_3+(wz_3\star w')z_3+(w\star w')z_2^3$
\end{itemize}
for any words $w$ and $w'$ in $\cH^{2,3}$, and then extended by $\ZZ$-bilinearity.
Let $\tau$ be an anti-automorphism on $\cH$ defined by $\tau(x)=y$ and $\tau(y)=x$.
Using this notation, we can now state the following theorem.
\begin{theorem}[{Hirose--Sato~\cite[Proposition~14]{HiroseSato2022+}}]\label{thm:HiroseSato2}
For $w_1, w_2\in\cH^{2,3}$, we have
\[
Z(w_1\tau(w_2))=Z(w_1\star w_2).
\]
\end{theorem}
By specializing Theorem~\ref{thm:HiroseSato2} to $w_1=z_2^{a-1}z_3z_2^{c-1}$ and $w_2=z_2^{b-1}z_3$, we recover Theorem~\ref{thm:HiroseSato1}.
Hirose and Sato derived Theorem~\ref{thm:HiroseSato2} from their \emph{block shuffle identity} (\cite[Theorem~10]{HiroseSato2022+}).

Hoffman conjectured in \cite{Hoffman1997} that any multiple zeta value can be expressed as a $\QQ$-linear combination of the elements of
\[
\{\zeta(k_1,\dots, k_r) \mid r\in\ZZ_{>0}, k_1,\dots, k_r\in\{2,3\}\}.
\]
This conjecture was proved by Brown in \cite{Brown2012}.
Although the linear independence of these elements remains open, they are called the \emph{Hoffman basis}.
Theorem~\ref{thm:HiroseSato2} provides a family for which the expansion in the Hoffman basis can be described explicitly, and, even more remarkably, it shows that the expansion can be taken with \emph{integer coefficients}.
In general, the coefficients are not integral; for instance,
\[
\zeta(3,1,4)=-\frac{2048}{4125}\zeta(2,2,2,2)-\frac{17}{275}\zeta(2,3,3)+\frac{2}{275}\zeta(3,2,3)+\frac{113}{275}\zeta(3,3,2).
\]

In contrast, Hirose, Maesaka, Seki, and Watanabe recently proved in \cite{HiroseMaesakaSekiWatanabe2025+}, together with an explicit algorithm, that any multiple zeta value can be expressed as a $\ZZ$-linear combination of
\[
\{\zeta(k_1,\dots, k_r) \mid r\in\ZZ_{>0}, k_1,\dots, k_r\geq 2\}.
\]
Their explicit expansions give rise to a family of relations among MZVs, which we refer to as the \emph{drop 1 relation}.
Since their expansions have integer coefficients, it is natural to ask whether Theorem~\ref{thm:HiroseSato2} can be viewed as an expansion arising from the drop 1 relation.
However, in general, expansions coming from the drop 1 relation involve MZVs with entries greater than $3$.
For instance,
\[
\zeta(3,1,4)=\zeta(5,3)-\zeta(4,4)+\zeta(2,4,2)-\zeta(2,3,3)-2\zeta(3,2,3)-\zeta(3,3,2).
\]
Therefore, it is not obvious whether the expansion in Theorem~\ref{thm:HiroseSato2}, in which only 2 and 3 appear, can be obtained from the drop 1 relation.
The main result of this paper is that this is indeed the case.

A subring $\cH^{\geq 2}\subset\cH^0$ is defined by $\cH^{\geq 2}\coloneqq\ZZ+z_2\ZZ\langle x,z_2\rangle=\ZZ\langle z_2, z_3, z_4,\dots\rangle$.
A linear map $\cD\colon\cH^0\to\cH^{\geq 2}$, called the \emph{drop 1 operator}, is defined in Definition~\ref{def:drop1operator}, and its properties are summarized in Theorem~\ref{thm:drop1}.
\begin{theorem}\label{thm:main}
For $w_1, w_2\in\cH^{2,3}$, we have
\[
\cD(w_1\tau(w_2))=w_1\star w_2.
\]
\end{theorem}
Not only does this theorem yield a new proof of Theorem~\ref{thm:HiroseSato2}, but it also implies, by Theorem~\ref{thm:drop1}~\ref{it:HMSW1}, that the same relations hold for \emph{multiple zeta diamond values}.
We omit the definition of $\zeta^{\diamondsuit}(k_1,\dots, k_r)$ here; see \cite[Definition~2.3]{HiroseMaesakaSekiWatanabe2025+}.
A $\ZZ$-linear map $Z^{\diamondsuit}\colon\cH^0\to\QQ^{\NN}$ is defined by $Z^{\diamondsuit}(z_{k_1}\cdots z_{k_r})=\zeta^{\diamondsuit}(k_1,\dots, k_r)$ and $Z^{\diamondsuit}(1)=1$.

\begin{corollary}\label{cor:diamond-lift}
For $w_1, w_2\in\cH^{2,3}$, we have
\[
Z^{\diamondsuit}(w_1\tau(w_2))=Z^{\diamondsuit}(w_1\star w_2).
\]
\end{corollary}
For example, the relation corresponding to \eqref{eq:Hoffman-conj}, 
\[
\zeta^{\diamondsuit}(3,\{2\}^{c-1},1,2)=\zeta^{\diamondsuit}(\{2\}^{c+2})+2\zeta^{\diamondsuit}(3,3,\{2\}^{c-1})
\]
means that, for any positive integer $N$,
\begin{align*}
&\sum_{0<n_1<n_2<\cdots<n_{c}<n_{c+1}<n_{c+2}<N}\frac{1}{n_1^3n_2^2\cdots n_{c}^2n_{c+1}^{}n_{c+2}^2}\\
&+\sum_{0<n_1<n_2<\cdots<n_{c}<n_{c+1}\leq n_{c+2}<N}\frac{1}{n_1^3n_2^2\cdots n_{c}^2(N-n_{c+1}^{})n_{c+2}^2}\\
&=\sum_{0<n_1<\cdots < n_{c+2}<N}\frac{1}{n_1^2\cdots n_{c+2}^2}+2\sum_{0<n_1<n_2<n_3<\cdots<n_{c+1}<N}\frac{1}{n_1^3n_2^3n_3^2\cdots n_{c+1}^2}
\end{align*}
holds.
Letting $N\to\infty$ yields \eqref{eq:Hoffman-conj}.
On the other hand, if we take $N=p$ for a prime $p$ and take modulo $p$, then we have the following relation among \emph{finite multiple zeta values} (FMZVs):
\[
\zeta^{}_{\cA}(3,\{2\}^{c-1},3)+2\zeta^{}_{\cA}(3,3,\{2\}^{c-1})=0.
\]
Here, for positive integers $k_1,\dots, k_r$, the FMZV $\zeta^{}_{\cA}(k_1,\dots, k_r)$ is defined as an element of $\cA\coloneqq\left.\prod_p\ZZ/p\ZZ\right/\bigoplus_p\ZZ/p\ZZ$ by
\[
\zeta^{}_{\cA}(k_1,\dots, k_r)\coloneqq (\zeta^{}_p(k_1,\dots,k_r)\bmod{p})_p,
\]
where $p$ runs over all prime numbers, and
\[
\zeta^{}_p(k_1,\dots,k_r)\coloneqq\sum_{0<n_1<\cdots<n_r<p}\frac{1}{n_1^{k_1}\cdots n_r^{k_r}}.
\]
This definition is due to Kaneko and Zagier (\cite{KanekoZagier2025+}).
We have used a well-known formula: $\zeta^{}_{\cA}(\{2\}^{c+2})=0$.

More generally, applying the same procedure to Corollary~\ref{cor:diamond-lift} (see Section~\ref{sec:finite} for more details), we have the following.
For a word $w=z_{k_1}\cdots z_{k_r}$, we set $\overleftarrow{w}\coloneqq z_{k_r}\cdots z_{k_1}$, and $\wt(w)$ denotes the total degree of $w$, that is, $\wt(w)=k_1+\cdots + k_r$.
A $\ZZ$-linear map $Z_{\cA}\colon \ZZ+y\cH=\ZZ\langle z_1,z_2,\dots\rangle\to\cA$ is defined by $Z_{\cA}(z_{k_1}\cdots z_{k_r})=\zeta^{}_{\cA}(k_1,\dots, k_r)$ and $Z_{\cA}(1)=1$.
\begin{corollary}\label{cor:FMZV}
For words $w_1, w_2\in\cH^{2,3}$, we have
\[
(-1)^{\wt(w_2)}Z_{\cA}(w_1\overleftarrow{w_2})=Z_{\cA}(w_1\star w_2).
\]
\end{corollary}
This is a family of relations among FMZVs in which only 2 and 3 occur as entries.
Since the shuffle relation for FMZVs $(-1)^{\wt(w_2)}Z_{\cA}(w_1\overleftarrow{w_2})=Z_{\cA}(w_1\sh w_2)$ holds (see \cite{KanekoZagier2025+,Ono2017}), Corollary~\ref{cor:FMZV} can also be rewritten as
\[
Z_{\cA}(w_1\sh w_2-w_1\star w_2)=0
\]
which is of the form of a certain double shuffle relation.
\section*{Acknowledgements}
The author would like to thank Professor Minoru Hirose for kindly answering his questions about the previous works by Hirose and Sato.
This research was inspired by a question that came to mind when the author got married, and the author would like to thank his wife, Yui, for marrying him.
\section{The drop 1 operator}
We introduce the following notation to define the main operator.
Let $t$ be a positive integer, let $\bc=(c_1,\dots, c_{2t})$ be a tuple of positive integers, and let $S\subset\ZZ$.
\begin{itemize}[leftmargin=2.5em]
\item $\#A$ denotes the cardinality of a finite set $A$.
\item $[n]\coloneqq\{1,2,\dots,n\}$ for $n\in\ZZ_{\geq 0}$. $[0]=\varnothing$.
\item $[2t]_{\bc}^1\coloneqq\{i\in[2t] \mid c_i=1\}.$
\item $[2t]_{\bc}^{>1}\coloneqq\{i\in[2t] \mid c_i>1\}.$
\item $\cP^{\text{eo}}(S)\coloneqq\{A\subset S \mid i\in A\cap 2\ZZ \Longleftrightarrow i+1\in A\setminus 2\ZZ\}.$
\item $\cP^{\text{oe}}(S)\coloneqq\{A\subset S \mid i\in A\setminus 2\ZZ \Longleftrightarrow i+1\in A\cap 2\ZZ\}.$
\item $\cP^{\smallxcancel{\text{eo}}}(S)\coloneqq\{A\subset S \mid i\in A\cap 2\ZZ \Longrightarrow i+1\not\in A\setminus 2\ZZ\}.$
\item $\cP^{\smallxcancel{\text{oe}}}(S)\coloneqq\{A\subset S \mid i\in A\setminus 2\ZZ \Longrightarrow i+1\not\in A\cap 2\ZZ\}.$
\item $\bc_{(-A)}\coloneqq (c_{i_1},\dots, c_{i_r})$, where $[2t]\setminus A=\{i_1,\dots, i_r\}$ with $i_1<\cdots<i_r$, for $A\subset[2t]$.
\item $\bc_{(-\varnothing)}=\bc$, $\bc_{(-[2t])}=\varnothing$.
\item $\delta_{i\in B}\coloneqq\begin{cases}1 & \text{if } i\in B, \\ 0 & \text{otherwise}\end{cases}$ for a set $B$.
\item $\bc_{(-A)}-\bdelta_B\coloneqq(c_{i_1}-\delta_{i_1\in B},\dots, c_{i_r}-\delta_{i_r\in B})$ for the above $\bc_{(-A)}$ and $B\subset [2t]$.
\end{itemize}

\begin{definition}[Hirose--Maesaka--Seki--Watanabe~\cite{HiroseMaesakaSekiWatanabe2025+}]\label{def:drop1operator}
We define the \emph{drop 1 operator} $\cD\colon\cH^0\to\cH^{\geq 2}$ to be a $\ZZ$-linear map given as follows.
It suffices to define $\cD(w)$ for each word $w\in\cH^0$.
For $\bc=(c_1,\dots, c_{2t})$, we write $\fD(\bc)=\cD(y^{c_1}x^{c_2}\cdots y^{c_{2t-1}}x^{c_{2t}})$.
The image $\fD(\bc)$ is defined inductively as follows:
\begin{align*}
\fD(\bc)&=\sum_{\substack{A\in\cP^{\text{eo}}([2t]^1_{\bc}) \\ B\in\cP^{\smallxcancel{\text{eo}}}([2t]^{>1}_{\bc}) \\ \#A+\#B\geq 1}}(-1)^{\#B-1}\fD(\bc_{(-A)}-\bdelta_B)x^{\#A+\#B}\\
&\quad+\sum_{\substack{A\in\cP^{\text{eo}}([2t]^1_{\bc}) \\ B\in\cP^{\smallxcancel{\text{eo}}}([2t]^{>1}_{\bc}) \\ \#A+\#B\geq 2}}(-1)^{\#B-1}\fD(\bc_{(-A)}-\bdelta_B)z_{\#A+\#B}\\
&\quad+\sum_{\substack{A\in\cP^{\text{oe}}([2t]^1_{\bc}) \\ B\in\cP^{\smallxcancel{\text{oe}}}([2t]^{>1}_{\bc}) \\ \#A+\#B\geq 2}}(-1)^{\#B}\fD(\bc_{(-A)}-\bdelta_B)z_{\#A+\#B}.
\end{align*}
Here the initial condition is $\fD(\varnothing)=\cD(1)=1$.
\end{definition}
\begin{theorem}[Hirose--Maesaka--Seki--Watanabe~\cite{HiroseMaesakaSekiWatanabe2025+}]
The drop 1 operator $\cD$ satisfies the following properties$:$
\label{thm:drop1}
\begin{enumerate}[label=\textnormal{(\roman*)}]
\item\label{it:HMSW1} $\{w-\cD(w) \mid w\in \cH^0\}=\Ker(Z^{\diamondsuit})\subset\Ker(Z)$,
\item\label{it:HMSW2} $\cD(w)=w$ for $w\in\cH^{\geq 2}$,
\item\label{it:HMSW3} $\cD(w)=\cD(\tau(w))$ for $w\in \cH^0$.
\end{enumerate}
\end{theorem}
In particular, the identity $Z(w)=Z(\cD(w))$ implies that any MZV can be expanded as a $\ZZ$-linear combination of $\{\zeta(k_1,\dots, k_r) \mid r\in\ZZ_{>0}, k_1,\dots, k_r\geq 2\}$.
Theorem~\ref{thm:drop1}~\ref{it:HMSW3} is not mentioned explicitly in \cite{HiroseMaesakaSekiWatanabe2025+}, but it follows easily by induction from the definition; it can also be proved as follows.
As noted in \cite[the last sentence of Section~3]{HiroseMaesakaSekiWatanabe2025+}, one has $w-\tau(w)\in\Ker(Z^{\diamondsuit})$.
Hence, by \ref{it:HMSW1}, there exists $v\in \cH^0$ such that $w-\tau(w)=v-\cD(v)$.
Therefore, by \ref{it:HMSW2}, we have $\cD(w-\tau(w))=\cD(v-\cD(v))=0$.
Although it is beyond the scope of this paper, we state the following conjecture.
\begin{conjecture}
For words $w, w'$ in $\cH^0$, if $\cD(w)=\cD(w')$, then either $w'=w$ or $w'=\tau(w)$.
\end{conjecture}
The corresponding statement with $\cD$ replaced by $Z$ was conjectured by Borwein, Bradley, Broadhurst, and Lison\v{e}k in \cite{BBBL}.
By \ref{it:HMSW1}, if their conjecture is true, then the above conjecture is also true.
\section{Proof of Theorem~\ref{thm:main}}
\begin{proof}[Proof of Theorem~$\ref{thm:main}$]
It suffices to consider the case where $w_1$ and $w_2$ are words, and we prove the claim by induction on $\wt(w_1)+\wt(w_2)$.
If $w_1=1$ or $w_2=1$, then the claim is clear from Theorem~\ref{thm:drop1}~\ref{it:HMSW2} and \ref{it:HMSW3}.
For non-negative integers $r$, $s$, and positive integers $a_1,\dots, a_r$, $b_1, \dots, b_s$, $c$, it suffices to consider the case
\[
w_1\tau(w_2)=z_2^{a_1-1}z_3\cdots z_2^{a_r-1}z_3z_2^{c-1}yz_2^{b_1}\cdots yz_2^{b_s},
\]
without loss of generality.
Let
\begin{align*}
\bc&\coloneqq [a_1,\dots, a_r;c;b_1,\dots, b_s]\\
&\coloneqq(\{1\}^{2a_1-1},2,\dots,\{1\}^{2a_r-1},2,\{1\}^{2c-2},2,\{1\}^{2b_1-1},\dots, 2,\{1\}^{2b_s-1}).
\end{align*}
Then $\cD(w_1\tau(w_2))=\fD(\bc)$ holds.
When $r=0$, read as
\[
\bc=[\varnothing;c;b_1,\dots, b_s]=(\{1\}^{2c-2},2,\{1\}^{2b_1-1},\dots, 2,\{1\}^{2b_s-1});
\]
similarly for $s=0$.
For a set $A\subset\ZZ$ and an integer $l$, we use the notation $A+l\coloneqq\{a+l\mid a\in A\}$.
Let $l_1\coloneqq a_1$, $l_2\coloneqq a_1+a_2$, $\dots$, $l_r\coloneqq a_1+\cdots+a_r$, $m_1\coloneqq b_1$, $m_2\coloneqq b_1+b_2$, $\dots$, $m_s\coloneqq b_1+\cdots+b_s$, and $t\coloneqq l_r+m_s+c-1$.
When $r=0$ (resp.~$s=0$), $l_r\coloneqq 0$ (resp.~$m_s\coloneqq0$).
Since we are considering the induction step, we may assume that $t>0$.
Note that $\wt(w_1\tau(w_2))$ equals $2t+r+s$, not $t$.
In this situation, $[2t]^{>1}_{\bc}=\{2l_1,2l_2,\dots, 2l_r\}\cup\bigl(\{0,2m_1, 2m_2, \dots, 2m_{s-1}\}+(2l_r+2c-1)\bigr)$ and $\cP^{\smallxcancel{\text{oe}}}([2t]^{>1}_{\bc})=\cP([2t]^{>1}_{\bc})$ hold, where $\cP([2t]^{>1}_{\bc})$ denotes the power set of $[2t]^{>1}_{\bc}$.

We consider the case $c>1$.
Then $\cP^{\smallxcancel{\text{eo}}}([2t]^{>1}_{\bc})=\cP([2t]^{>1}_{\bc})$ holds.
For $A\in\cP^{\text{eo}}([2t]^1_{\bc})$, decompose it as $A=A_1\cup A_2 \cup A_3$, where
\[
A_1\coloneqq A \cap [2l_r-1],\quad A_2\coloneqq A\cap ([2l_r+2c-2]\setminus [2l_r]),\quad A_3\coloneqq A\setminus [2l_r+2c-1].
\]
We define a map $\iota_1\colon\cP^{\text{eo}}([2t]^1_{\bc})\to\cP^{\text{oe}}([2t]^1_{\bc})$ by
\[
\iota_1(A)\coloneqq (A_1-1)\cup (A_2-1)\cup\{2l_r+2c-3,2l_r+2c-2\}\cup(A_3+1)
\]
and a map $\iota_2\colon\cP^{\text{eo}}([2t]^1_{\bc})\to\cP^{\text{oe}}([2t]^1_{\bc})$ by
\[
\iota_2(A)\coloneqq (A_1-1)\cup (A_2-1)\cup(A_3+1).
\]
A bijection $\cP^{\text{eo}}([2t]^1_{\bc})\sqcup\cP^{\text{eo}}([2t]^1_{\bc})\to\cP^{\text{oe}}([2t]^1_{\bc})$ is obtained by sending an element $A$ in the first copy of $\cP^{\text{eo}}([2t]^1_{\bc})$ to $\iota_1(A)$, and an element $A$ in the second copy of $\cP^{\text{eo}}([2t]^1_{\bc})$ to $\iota_2(A)$.
For $A\in\cP^{\text{eo}}([2t]^1_{\bc})$, we have $\bc_{(-A)}=\bc_{(-\iota_2(A))}$.
Therefore, if we compute the defining recurrence for the drop 1 operator using this bijection, every term in the second sum is canceled by a term in the third sum, and we are left with
\begin{align*}
\fD(\bc)&=\sum_{\substack{A\in\cP^{\text{eo}}([2t]^1_{\bc}) \\ B\in\cP([2t]^{>1}_{\bc}) \\ \#A+\#B\geq 1}}(-1)^{\#B-1}\biggl(\fD(\bc_{(-A)}-\bdelta_B)-\fD(\bc_{(-\iota_1(A))}-\bdelta_B)z_2\biggr)x^{\#A+\#B}\\
&\quad+\fD(\bc_{(-\{2l_r+2c-3,2l_r+2c-2\})})z_2.
\end{align*}
For $A\in\cP^{\text{eo}}([2t]^1_{\bc})$ or $\cP^{\text{oe}}([2t]^1_{\bc})$, there exist $a'_1$, $\dots$, $a'_r$, $b'_1$, $\dots$, $b'_s$, and $c'$ such that
\[
\bc_{(-A)}=[a'_1,\dots, a'_r;c';b'_1,\dots, b'_s].
\]
Here $a'_1\in[a_1]$, $\dots$, $a'_r\in[a_r]$, $b'_1\in[b_1]$, $\dots$, $b'_s\in[b_s]$, $c'\in[c]$, and $a'_1+\cdots+a'_r+b'_1+\cdots+b'_s+c'=t+1-\#A/2$.
Moreover, upon adding $-\bdelta_B$, addition occurs between the corresponding components; for example, when
\begin{align*}
\bc&=[2,1,3;3;4,1,1,2]\\
&=(\stackrel{1}{1},\stackrel{2}{1},\stackrel{3}{1},\stackrel{\color{red}{4}}{2},\stackrel{5}{1},\stackrel{\color{red}{6}}{2},\stackrel{7}{1},\stackrel{8}{1},\stackrel{9}{1},\stackrel{10}{1},\stackrel{11}{1},\stackrel{\color{red}{12}}{2},\stackrel{13}{1},\stackrel{14}{1},\stackrel{15}{1},\stackrel{16}{1},\stackrel{\color{red}{17}}{2},\stackrel{18}{1},\stackrel{19}{1},\stackrel{20}{1},\stackrel{21}{1},\stackrel{22}{1},\stackrel{23}{1},\stackrel{24}{1},\stackrel{\color{red}{25}}{2},\stackrel{26}{1},\stackrel{\color{red}{27}}{2},\stackrel{28}{1},\stackrel{\color{red}{29}}{2},\stackrel{30}{1},\stackrel{31}{1},\stackrel{32}{1}),
\end{align*}
\[
A=\{2,3,8,9,14,15,18,19,22,23,30,31\}\in\cP^{\text{eo}}([32]^1_{[2,1,3;3;4,1,1,2]}),
\]
and
\[
B=\{4,17,27\}\subset[32]^{>1}_{[2,1,3;3;4,1,1,2]}=\{4,6,12,17,25,27,29\},
\]
we have
\begin{align*}
&[2,1,3;3;4,1,1,2]_{(-\{2,3,8,9,14,15,18,19,22,23,30,31\})}=[1,1,2;2;2,1,1,1]\\
&=(\stackrel{1}{1},\stackrel{\color{red}{4}}{2},\stackrel{5}{1},\stackrel{\color{red}{6}}{2},\stackrel{7}{1},\stackrel{10}{1},\stackrel{11}{1},\stackrel{\color{red}{12}}{2},\stackrel{13}{1},\stackrel{16}{1},\stackrel{\color{red}{17}}{2},\stackrel{20}{1},\stackrel{21}{1},\stackrel{24}{1},\stackrel{\color{red}{25}}{2},\stackrel{26}{1},\stackrel{\color{red}{27}}{2},\stackrel{28}{1},\stackrel{\color{red}{29}}{2},\stackrel{32}{1})
\end{align*}
and
\begin{align*}
&[2,1,3;3;4,1,1,2]_{(-\{2,3,8,9,14,15,18,19,22,23,30,31\})}-\bdelta_{\{4,17,27\}}\\
&=[1+1,2;2+2;1+1,1]=[2,2;4;2,1]\\
&=(\stackrel{1}{1},\stackrel{\color{blue}{4}}{1},\stackrel{5}{1},\stackrel{\color{red}{6}}{2},\stackrel{7}{1},\stackrel{10}{1},\stackrel{11}{1},\stackrel{\color{red}{12}}{2},\stackrel{13}{1},\stackrel{16}{1},\stackrel{\color{blue}{17}}{1},\stackrel{20}{1},\stackrel{21}{1},\stackrel{24}{1},\stackrel{\color{red}{25}}{2},\stackrel{26}{1},\stackrel{\color{blue}{27}}{1},\stackrel{28}{1},\stackrel{\color{red}{29}}{2},\stackrel{32}{1}).
\end{align*}
Therefore, for any $A, B$ with $\#A+\#B\geq 1$, the induction hypothesis implies that there exist words $v_1, v_2$ in $\cH^{2,3}$ such that $\bc_{(\iota_1(-A))}-\bdelta_B$ corresponds to $v_1\tau(v_2)$, and we have
\[
\fD(\bc_{(-A)}-\bdelta_B)-\fD(\bc_{(-\iota_1(A))}-\bdelta_B)z_2=\cD(v_1z_2\tau(v_2))-\cD(v_1\tau(v_2))z_2=v_1z_2\star v_2-(v_1\star v_2)z_2=0.
\]
Since we are in the case $c>1$, we can write $w_1=wz_2$ with a certain word $w$ in $\cH^{2,3}$.
Then we obtain
\begin{align*}
\fD(\bc)&=\fD(\bc_{(-\{2l_r+2c-3,2l_r+2c-2\})})z_2=\fD([a_1,\dots, a_r;c-1;b_1,\dots, b_s])z_2=\cD(w\tau(w_2))z_2\\
&=(w\star w_2)z_2=wz_2\star w_2=w_1\star w_2.
\end{align*}

It remains to consider the case $c=1$.
In this case, we have $\cP^{\smallxcancel{\text{eo}}}([2t]^{>1}_{\bc})=\{B\in\cP([2t]^{>1}_{\bc})\mid\{2l_r,2l_r+1\}\not\subset B\}$.
For $A\in\cP^{\text{eo}}([2t]^1_{\bc})$, decompose it as $A=A_1\cup A_3$, where
\[
A_1\coloneqq A \cap [2l_r-1],\quad A_3\coloneqq A\setminus [2l_r+1].
\]
We define a map $\iota_2\colon\cP^{\text{eo}}([2t]^1_{\bc})\to\cP^{\text{oe}}([2t]^1_{\bc})$ by
\[
\iota_2(A)\coloneqq (A_1-1)\cup(A_3+1).
\]
Since $\iota_2$ is a bijection in this case, we compute 
\begin{align*}
\fD(\bc)&=\sum_{\substack{A\in\cP^{\text{eo}}([2t]^1_{\bc}) \\ B\in\cP([2t]^{>1}_{\bc}) \\ \{2l_r,2l_r+1\}\not\subset B \\ \#A+\#B\geq 1}}(-1)^{\#B-1}\fD(\bc_{(-A)}-\bdelta_B)x^{\#A+\#B}\\
&\quad+\sum_{\substack{A\in\cP^{\text{oe}}([2t]^1_{\bc}) \\ B\in\cP([2t]^{>1}_{\bc}) \\ \{2l_r,2l_r+1\}\subset B}}(-1)^{\#B}\fD(\bc_{(-A)}-\bdelta_B)z_{\#A+\#B}\\
&=\sum_{\substack{A\in\cP^{\text{eo}}([2t]^1_{\bc}) \\ B\in\cP([2t]^{>1}_{\bc}) \\ 2l_r,2l_r+1\not\in B \\ \#A+\#B\geq 1}}(-1)^{\#B-1}F(\bc,A,B)x^{\#A+\#B}\\
&\quad+\fD(\bc-\bdelta_{\{2l_r\}})x+\fD(\bc-\bdelta_{\{2l_r+1\}})x+\fD(\bc-\bdelta_{\{2l_r,2l_r+1\}})z_2,
\end{align*}
where
\begin{align*}
F(\bc,A,B)&\coloneqq\fD(\bc_{(-A)}-\bdelta_B)-\fD(\bc_{(-A)}-\bdelta_{B\cup\{2l_r\}})x-\fD(\bc_{(-A)}-\bdelta_{B\cup\{2l_r+1\}})x\\
&\quad -\fD(\bc_{(-A)}-\bdelta_{B\cup\{2l_r,2l_r+1\}})z_2.
\end{align*}
Since we may assume that $r\geq 1$ and $s\geq 1$, when $2l_r, 2l_r+1\not\in B$, the same argument as in the previous case shows that $\bc_{(-A)}-\bdelta_B$ takes the form
\[
\bc_{(-A)}-\bdelta_B=[a'_1,\dots,a'_{r'};1;b'_1,\dots,b'_{s'}],
\]
where $r'$, $s'$, $a'_1$, $\dots$, $a'_{r'}$, and $b'_1$, $\dots$, $b'_{s'}$ are suitable positive integers.
Therefore, for any $A, B$ with $2l_r, 2l_r+1\not\in B$ and $\#A+\#B\geq 1$, the induction hypothesis implies that there exist words $v_1, v_2$ in $\cH^{2,3}$ such that
\begin{align*}
\fD(\bc_{(-A)}-\bdelta_B)&=\cD(v_1z_3\tau(v_2z_3))\\
&=v_1z_3\star v_2z_3\\
&=(v_1\star v_2z_3)z_3+(v_1z_3\star v_2)z_3+(v_1\star v_2)z_2^3\\
&=(v_1z_2\star v_2z_3)x+(v_1z_3\star v_2z_2)x+(v_1z_2\star v_2z_2)z_2\\
&=\fD(\bc_{(-A)}-\bdelta_{B\cup\{2l_r\}})x+\fD(\bc_{(-A)}-\bdelta_{B\cup\{2l_r+1\}})x\\
&\quad+\fD(\bc_{(-A)}-\bdelta_{B\cup\{2l_r,2l_r+1\}})z_2,
\end{align*}
that is, $F(\bc,A,B)=0$.
Since we are in the case $c=1$, $r\geq 1$, and $s\geq 1$, we can write $w_1=wz_3$ and $w_2=w'z_3$ with certain words $w, w'$ in $\cH^{2,3}$.
Then we obtain
\begin{align*}
\fD(\bc)&=\fD(\bc-\bdelta_{\{2l_r\}})x+\fD(\bc-\bdelta_{\{2l_r+1\}})x+\fD(\bc-\bdelta_{\{2l_r,2l_r+1\}})z_2\\
&=\fD([a_1,\dots, a_{r-1};a_r+1;b_1,\dots,b_s])x+\fD([a_1,\dots,a_r;1+b_1;b_2,\dots,b_s])x\\
&\quad+\fD([a_1,\dots,a_{r-1};a_r+1+b_1;b_2,\dots,b_s])z_2\\
&=\cD(wz_2\tau(w'z_3))x+\cD(wz_3\tau(w'z_2))x+\cD(wz_2\tau(w'z_2))z_2\\
&=(wz_2\star w'z_3)x+(wz_3\star w'z_2)x+(wz_2\star w'z_2)z_2\\
&=(w\star w'z_3)z_3+(wz_3\star w')z_3+(w\star w')z_2^3\\
&=wz_3\star w'z_3=w_1\star w_2.
\end{align*}
This completes the proof.
\end{proof}
In relation to the main result, we can consider the following conjecture.
\begin{conjecture}\label{conj:inverse}
For any word $w$ in $\cH^0$ that cannot be written in the form $w_1\tau(w_2)$ for words $w_1, w_2$ in $\cH^{2,3}$, we have $\cD(w)\not\in\cH^{2,3}$.
\end{conjecture}
Hirose and Sato conjectured in \cite[Remark~15]{HiroseSato2022+} that the expansion of an MZV in the Hoffman basis has integer coefficients only for the expansion provided by Theorem~\ref{thm:HiroseSato2}.
If their conjecture is true, then Conjecture~\ref{conj:inverse} also follows.
\section{Proof of Corollary~\ref{cor:FMZV}}\label{sec:finite}
For the definitions of $\zeta^{\diamondsuit}_N$, $\zeta^{\flat}_N$, and $\zeta^{\diamondsuit\flat}_N$, see \cite[Definitions~2.3, 3.4, and 3.9]{HiroseMaesakaSekiWatanabe2025+}.
Here we take $N=p$.
\begin{lemma}\label{lem:diamond-modp}
Let $p$ be a prime and $s$ be a positive integer.
Let $a_1,\dots, a_s, b_1,\dots, b_s$ be positive integers.
Set $l\coloneqq a_1+\cdots+a_s-s$.
Then, we have
\[
\zeta^{\diamondsuit}_p(\{1\}^{a_1-1},b_1+1,\dots, \{1\}^{a_s-1},b_s+1)\equiv(-1)^l\zeta^{}_p(a_1+b_1,\dots, a_s+b_s)\pmod{p}.
\]
\end{lemma}
\begin{proof}
By \cite[Corollary~3.10]{HiroseMaesakaSekiWatanabe2025+} and \cite[Theorem~1.3]{MaesakaSekiWatanabe2024+} (= \cite[Theorem~3.5]{HiroseMaesakaSekiWatanabe2025+}), we compute
\begin{align*}
&\zeta^{\diamondsuit}_p(\{1\}^{a_1-1},b_1+1,\dots, \{1\}^{a_s-1},b_s+1)\\
&=\zeta^{\diamondsuit\flat}_p(\{1\}^{a_1-1},b_1+1,\dots, \{1\}^{a_s-1},b_s+1)\\
&=\sum_{\substack{0<n_{i,1}\leq\cdots\leq n_{i,a_i}\leq m_{i,1}\leq\cdots\leq m_{i,b_i}<p \ (i\in[s]) \\ m_{i,b_i}<n_{i+1,1} \ (i\in[s-1])}}\prod_{i\in[s]}\frac{1}{(p-n_{i,1})\cdots(p-n_{i,a_i})}\frac{1}{m_{i,1}\cdots m_{i,b_i}}\\
&\equiv(-1)^l\sum_{\substack{0<n_{i,1}\leq\cdots\leq n_{i,a_i}\leq m_{i,1}\leq\cdots\leq m_{i,b_i}<p \ (i\in[s]) \\ m_{i,b_i}<n_{i+1,1} \ (i\in[s-1])}}\prod_{i\in[s]}\frac{1}{(p-n_{i,1})n_{i,2}\cdots n_{i,a_i}}\frac{1}{m_{i,1}\cdots m_{i,b_i}}\pmod{p}\\
&=(-1)^l\zeta_p^{\flat}(a_1+b_1,\dots, a_s+b_s)=(-1)^l\zeta^{}_p(a_1+b_1,\dots,a_s+b_s).\qedhere
\end{align*}
\end{proof}
\begin{proof}[Proof of Corollary~$\ref{cor:FMZV}$]
Corollary~\ref{cor:diamond-lift} is an identity in $\QQ^{\NN}$.
By extracting the sequence of $p$-components over primes $p$ and projecting it to an identity in $\cA$, we have the desired equality.
For this argument, it remains to show that
\begin{align*}
&\zeta^{\diamondsuit}_p(\{2\}^{a_1-1},3,\dots,\{2\}^{a_r-1},3,\{2\}^{c-1},1,\{2\}^{b_1},\dots, 1,\{2\}^{b_s})\\
&\equiv(-1)^s\zeta^{}_p(\{2\}^{a_1-1},3,\dots,\{2\}^{a_r-1},3,\{2\}^{c-1},3,\{2\}^{b_1-1},\dots, 3,\{2\}^{b_s-1})\pmod{p}.
\end{align*}
This follows immediately from Lemma~\ref{lem:diamond-modp}.
\end{proof}

\end{document}